\documentclass[reqno]{amsart}
\usepackage[utf8]{inputenc}
\usepackage{amssymb,amsthm,amsfonts}
\usepackage{amsmath}
\usepackage{tikz}
\usepackage{graphicx}
\usepackage{psfrag}
\usepackage{cancel}
\usepackage{hyperref}
\usepackage{mathrsfs}
\usepackage{microtype}
\newtheorem{definition}{Definition}[section]
\newtheorem{theorem}{Theorem}[section]
\newtheorem{lemma}[theorem]{Lemma}
\newtheorem{corollary}[theorem]{Corollary}

\newtheorem{remark}{Remark}[section]

\numberwithin{equation}{section}

\theoremstyle{plain}

\newcommand{\nt}{\ensuremath{\mathbb{N}}}
\newcommand{\R}{\ensuremath{\mathbb{R}}}
\newcommand{\C}{\ensuremath{\mathbb{C}}}
\newcommand{\crit}{\operatorname{Crit}}

\newcommand{\dist}{\operatorname{dist}}
\newcommand{\diam}{\operatorname{diam}}

\renewcommand{\mod}{\operatorname{mod}}

\newcommand{\veps}{\varepsilon}

\numberwithin{equation}{section}

\newcommand{\cA}{{\mathcal A}}

\newcommand{\cP}{{\mathcal P}}

\newcommand{\cH}{{\mathcal H}}
\newcommand{\cR}{{\mathcal R}}

\newcommand{\RR}{{\Bbb R}}
\newcommand{\TT}{{\Bbb T}}
\newcommand{\ZZ}{{\Bbb Z}}
\newcommand{\NN}{{\Bbb N}}
\newcommand{\DD}{{\Bbb D}}
\newcommand{\CC}{{\Bbb C}}
\newcommand{\HH}{{\Bbb H}}
\newcommand{\QQ}{{\Bbb Q}}
\newcommand{\hol}{\Bbb H}

\title[Complex bounds for multicritical circle maps with bounded type ...]{Complex a priori bounds for multicritical circle maps with bounded type rotation number}

\author{Gabriela Estevez}
\address{Instituto de Matemática, Universidade Federal do Rio de Janeiro}
\curraddr{Av. Athos da Silveira Ramos 149, CEP 21945-909. Rio de Janeiro, RJ, Brazil}
\email{gaestevezja@gmail.com}

\author{Daniel Smania}
\address{Instituto de Ciências Matemáticas e de Computação, Universidade de S\~ao Paulo}
\curraddr{Av. Trab. São Carlense 400, CEP 13566-590. 
S\~ao Carlos, SP, Brazil}
\email{daniel.smania@gmail.com}

\author{Michael Yampolsky}
\address{Department of Mathematics, University of Toronto}
\curraddr{40 St George Street, Toronto, Ontario, Canada}
\email{yampol@math.toronto.edu}



\thanks{G.E. was partially supported by the Coordenação de Aperfeiçoamento de Pessoal de Nível Superior - Brasil (CAPES) - Finance Code 001. D.S. was partially supported by CNPq 306622/2019-0, CNPq 430351/2018-6 and FAPESP Projeto Temático 2017/06463-3. M.Y. was partially supported by NSERC Discovery Grant.}

\begin{document}

\begin{abstract} 
In this paper we study homeomorphisms of the circle with several critical points and bounded type rotation number. We prove complex a priori bounds for these maps. As an application, we get that bi-cubic circle maps with same bounded type rotation number are $C^{1+\alpha}$ rigid.
\end{abstract}

\maketitle

\section{Introduction}
Complex {\it a priori} bounds have emerged as a key analytic tool in one-dimensional dynamics. They provide the analytic foundation for the results in one-dimensional  Renormalization theory,  rigidity, density of hyperbolicity, and local connectivity of Julia sets and the Mandelbrot set.
Speaking informally, they are the bounds on the size of the domains of the analytic continuations of the first return maps corresponding to renormalizations of one-dimensional dynamical systems.
In this paper we prove complex {\it a priori}  bounds for  multicritical circle maps with rotation numbers of bounded type. This generalizes the results of \cite{Yam2019}, where they were
obtained under the assumption that the rotation number is a quadratic irrational, which is a particular case of bounded type.

Similarly to \cite{Yam2019}, we apply the bounds to the case of bi-cubic circle maps, and prove that such maps with irrational rotation numbers of bounded type are $C^{1+\alpha}$-rigid: that is, a topological conjugacy which maps critical points to critical points 
must be $C^{1+\alpha}$-regular.

\section{Preliminaries}
We will refer to the affine manifold $\TT=\RR/\ZZ$ as the circle, and will identify it as needed with the unit circle $S^1$ via the exponential map $x\mapsto e^{2\pi i x}$. For a homeomorphism $f:\TT\to\TT$ we will denote by $\rho(f)\in(0,1)$ its rotation number.

For $\alpha>0$, we let
$$G(\alpha)\equiv \left\{\frac{1}{\alpha}\right\}$$
be the Gauss map. Starting with $\alpha\in(0,1)$ we consider the orbit
$$\alpha_0\equiv \alpha, \, \dots  \,, \alpha_{n}=G(\alpha_{n-1}) \, ,\dots$$
It is finite if and only if $\alpha$ is rational, in which case we will end it at the last non-zero term. The numbers $a_n=[1/\alpha_n]$ for  $n\geq 0$ are the coefficients of the continued fraction expansion of $\alpha$ with positive terms (which is defined uniquely if and only if $\alpha\notin\QQ$). We will denote such continued fraction as
$$\alpha=[a_0,a_1,\ldots].$$
We say that $\alpha\in[0,1]\setminus\QQ$ is of a {\it type bounded by }$K\in\NN$ if $\sup a_i\leq K$. We will refer to the  union of such numbers for all $K\in\NN$ as irrationals of {\it bounded type}, note that this class coincides with Diophantine numbers of order $2$.
We will let $R_\alpha(x)\equiv x+\alpha \,\mod \ZZ$ denote the rigid rotation by angle $\alpha$.

Given two positive numbers $a$, $b$  we say that they are $C$-commensurable for $C>1$, and we denote it by $a \asymp_C b$, if
$$\frac{1}{C}\leq \frac{a}{b}\leq C.$$
We will say that $a$ and $b$ are universally commensurable, or simply commensurable, if the constant $C$ is universal. In that case we denote it by $a \asymp b$. Two sets in the plane are $C$-commensurable if their diameters are $C$-commensurable.

We will use $\diam(A)$ to denote the Euclidean diameter of a bounded set $A\subset \CC$. We let $\DD_r(z)$ be the open disk of radius $r$ centered at $z\in\CC$; $\DD$ will stand for the unit disk. $$\displaystyle U_r(A)=\cup_{z\in A}\DD_r(z)$$
will stand for the $r$-neighborhood of a set $A$. 
We denote
$$\dist(A,B)=\inf\{r>0\;|\; U_r(A)\cap B\neq \emptyset\}$$
the Euclidean distance between $A$ and $B$.

\begin{figure}[h]
\includegraphics[width=0.77\textwidth]{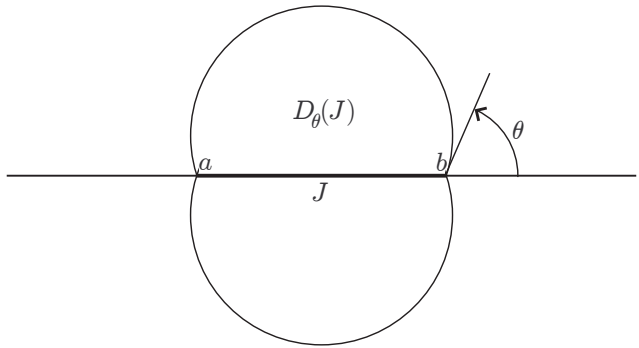}
\caption{\label{fig:poincare}A Poincar\'e neighborhood $D_\theta(J)$.}
\end{figure}

Given $J=(a,b)$, a subinterval in the real line, let
$\C_{J} =(\C\setminus \R) \cup J$.
Following Sullivan \cite{Su}, we let the {\emph Poincar\'e neighborhood} of $J$ of radius $r>0$ to be the  set of points in $\C_{J}$ such that their hyperbolic distance
in $\C_J$ to $J$ is less or equal to $r$. A Poincar\'e neighborhood is an $\RR$-symmetric union of two Euclidean disks with a  common chord $J$. If we denote the external angle between one of the boundary circles of such a neighborhood with $\RR$ by $\theta\in (0,\pi)$, then
$$r=r(\theta)=\log \cot (\theta/4).$$
It is more convenient for us to identify a Poincar\'e neighborhood by the external angle $\theta$ rather than the hyperbolic radius $r$, so we will use the notation $D_\theta(J)$ (see Figure~\ref{fig:poincare}).
Clearly, if $\theta_1<\theta_2$ then $D_{\theta_2}(J) \subset D_{\theta_1}(J)$ and so $r(\theta_2) <r(\theta_1)$.
We let $D_{\pi/2}(J)\equiv D(J)$; this is the Euclidean disk with diameter $J$.
We note that 
\begin{equation}\label{eq:diamD}
 \displaystyle \diam(D_{\theta}(J))= \left(\frac{1+\cos \theta}{\sin \theta} \right)|J|.
  \end{equation}

\subsection{Multicritical circle maps} We say that $f$ is a $C^{3}$ multicritical circle map if it is a $C^3$ orientation preserving circle homeomorphism with a finite number of {\it non-flat\/} critical points. That means that for each critical point $c$ there exist $d \in 2\nt +1$ ($d$ is called the criticality of $c$), a neighbourhood $W$ of $c$ and an orientation preserving diffeomorphism $\phi$ satisfying $\phi(c)=0$ such that for all $x \in W$, 
\begin{equation*}
f(x)=f(c)+\big(\phi(x)\big)^{d}. 
\end{equation*}
In the space of analytic maps, we say that an analytic multicritical circle map is just an analytic orientation preserving circle homeomorphism with a finite number of critical points.

We assume  that the rotation number of $f$ is irrational. By a result of Yoccoz \cite{Yoc1984}, this implies that $f$ is topologically conjugate with the rigid rotation by the angle $\rho(f)$. This clearly implies the existence of a unique ergodic $f$-invariant measure, which is the pullback of the Lebesgue measure by the conjugacy. We denote this measure by $\mu_f$. We define the \emph{signature} of a multicritical circle map $f$ to be the $(2N+2)$-tuple 
\[
(\rho(f) \,;N;\,d_0,d_1,\ldots,d_{N-1};\,\delta_0,\delta_1,\ldots,\delta_{N-1}),
\]
where $N$ is the number of critical points, $\rho(f)$ is the rotation number of $f$,  $d_i$ is the criticality of the critical point $c_i$, and $\delta_i=\mu_f[c_i,c_{i+1})$ (with the convention that $c_{N}=c_0$).

\subsection{Dynamical partitions}

Let $f$ be a multicritical circle map with irrational rotation number $\rho(f)$, and continued fraction expansion given by $\rho(f) \;=\; [a_{0} , a_{1} , \cdots ]$. Let us consider the continued  fraction convergents
$p_n/q_n$ obtained by truncating the expansion at level $n-1$, that is, $p_n/q_n\;=\;[a_0,a_1, \cdots ,a_{n-1}]$.  The sequence of denominators $\{q_n\}_{n\in\nt}$ satisfies the recursive formula 
\begin{equation*}
q_{0}=1, \hspace{0.4cm} q_{1}=a_{0}, \hspace{0.4cm} q_{n+1}=a_{n}\,q_{n}+q_{n-1} \hspace{0.3cm} \text{for all $n \geq 1$} .
\end{equation*}
Moreover, for $x \in S^1$ and $n \in \nt$, the iterates  $\{f^{q_n}(x)\}$ are closest returns of $x$ in the following sense: denote $I_n(x)=[x,f^{q_n}(x)]$ the arc of the circle connecting these two points and not containing $f^{q_{n+1}}(x)$. Then
$[x,f^{q_n}(x)]$ does not contain any iterates smaller than $q_n$ in the orbit of $x$.

For $a,b\in\TT$ we will denote $[a,b]$ the arc of the circle obtained as $\psi^{-1}(A)$ where $\psi$ is a conjugacy between $f$ and $R_{\rho(f)}$ and $A$ is the shorter of the two arcs connecting $\psi(a)$ with $\psi(b)$.

The collection of intervals
\[
 \mathcal{P}_n(x) \ = \ \left\{ f^{i}(I_n(x)):\;0\leq i\leq q_{n+1}-1 \right\} \;\bigcup\; 
 \left\{ f^{j}(I_{n+1}(x)):\;0\leq j\leq q_{n}-1 \right\} 
\]
is a partition of the circle by closed intervals intersecting only at their endpoints. It is called the {\it $n$-th dynamical partition\/} associated to the point $x$ (see \cite[Section 1.1, Lemma 1.3, page 26]{dMvS} or \cite[Appendix]{EdFG}). For each $n \in \nt$, we will refer to the intervals $I_n(x)$ and $I_{n+1}(x)$ as {\it fundamental intervals} of the dynamical partition $\mathcal{P}_n(x)$.

The dynamical partitions $\mathcal{P}_n(x)$ form a sequence of (non-strict) refinements: the intervals $I_n^j(x)$ for $0\leq j\leq q_{n+1}-1$ are subdivided by exactly $a_{n+1}$ intervals belonging to $\mathcal{P}_{n+1}(x)$ while the intervals $I_{n+1}^i(x)$ for each $0\leq i\leq q_n-1$ remain invariant, see Figure \ref{fig:partition} below.

\begin{figure}[!ht]
\centering
\psfrag{x}[][][1]{$x$} 
\psfrag{e}[][][1]{$f^{q_{n+1}}(x)$}
\psfrag{d}[][][1]{$f^{q_n}(x)$}
\psfrag{Pn}[][][1]{$\mathcal{P}_n(x)$} 
\psfrag{Pn+1}[][][1]{$\mathcal{P}_{n+1}(x)$}
\psfrag{In}[][][1]{$I_n$}
\psfrag{In+1}[][][1]{$I_{n+1}$} 
\psfrag{In+2}[][][1]{$I_{n+2}$} 
\psfrag{In+1qn}[][][1]{$I_{n+1}^{q_n}$} 
\psfrag{In+1qn+qn+1}[][][1]{$I_{n+1}^{q_n+q_{n+1}}$}
\psfrag{...}[][][1]{$\cdots$} 
\includegraphics[width=4.1in]{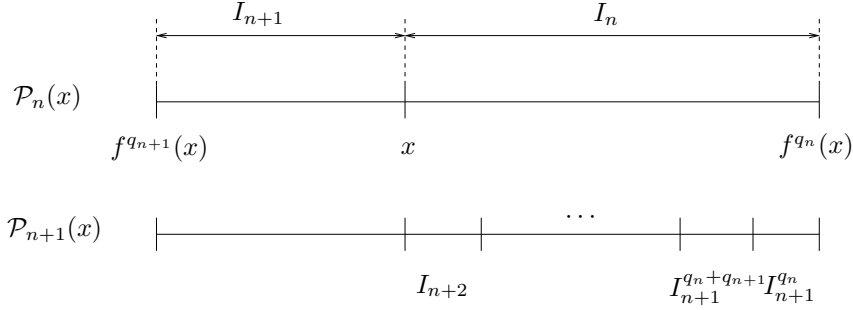}
  \caption{\label{fig:partition} Two consecutive dynamical partitions.}
\end{figure}

Following the convention introduced by Sullivan \cite{Su}, we say that for a map $f$ a quantity is ``beau'' (which translates as ``bounded and eventually universally (bounded)''), if it is bounded and the bound becomes universal (that is, independent of the map).

Let $c \in \crit(f)$, from now on we will consider $n$ bigger enough such that the adjacent intervals $I_{n+1}(c)$ and  $I_n(c)$, do not contain any other critical point of $f$. The following fundamental geometric control was obtained by Herman \cite{H} and Światek \cite{G} in the 1980's. A detailed proof of Theorem \ref{teobeau} can be found in \cite{EdFG}. 

\begin{theorem}[Real Bounds]\label{teobeau} 
Let $f$ be a multicritical circle map with irrational rotation number and $N$ critical points, and let $c$ be any of its critical points. Then there exists $n_0 \in \mathbb{N}$ such that for all $n\geq n_0$ the iterate $f^{q_{n+1}}|_{I_n(c)}$ is decomposed as
\[
 f^{q_{n+1}}|_{I_n(c)}= \psi_{m+1} \circ p_m \circ \psi_m\circ p_{m-1} \circ \dots \circ \psi_1 \circ p_0 \circ \psi_0,
\]
where $m \leq N +1$, $p_j(x)=x^{d_j}$ for $d_j$ an odd integer, and  each $\psi_j$ is an interval diffeomorphism with  beau distortion. 
\end{theorem}

An immediate corollary of Theorem \ref{teobeau} is the following result:

\begin{corollary}\label{corollarybeau}
Given $N \in \nt$ and $d>1$ let $\mathcal{F}_{N,d}$ be the family of multicritical circle maps with at most $N$ critical points whose maximum criticality is bounded by $d$. There exists a beau constant $C=C(N,d)>1$ with the following property: for any given $f \in \mathcal{F}_{N,d}$ and $c \in \crit(f)$ there exists $n_0 \in \mathbb{N}$ such that for all $n\geq n_0$ each pair of adjacent intervals $I,J \in \mathcal{P}_{n}(c)$ is $C$-commensurable.
\end{corollary}

\subsection{Renormalization of multicritical circle maps}

In this section, we recall the notion of \textit{multicritical commuting pair}, which is a generalization of \textit{critical commuting pair}  introduced in \cite{ORSS}. This notion will let us to define the renormalization of a multicritical circle map.

\begin{definition}\label{multcommpairs} A $C^3$ (or $C^{\infty}$) multicritical commuting pair is a pair $\zeta=(\eta,\xi)$ consisting of two $C^3$ orientation preserving interval homeomorphisms $\xi:I_\xi \rightarrow  \xi(I_\xi)$ and $\eta : I_\eta \rightarrow \eta(I_{\eta})$ 
satisfying:
 \begin{enumerate}
  \item $I_{\xi}=[\eta(0), 0]$ and $I_{\eta}=[0, \xi(0)]$ are compact intervals in the real line;
  \item the origin has odd integer criticality for $\eta$ and for $\xi$;
  \item $\xi$ and $\eta$ satisfy the commuting property: $(\eta \circ \xi)(0)=(\xi \circ \eta)(0) \neq 0$;
  \item $\xi$ and $\eta$, contain others  critical points (with odd integers criticalities) in theirs domains, $I_\xi$ and $I_\eta$;
  \item both $\xi$ and $\eta$, have homeomorphic extensions to some interval neighborhoods $V_{\xi}$ and $V_{\eta}$, of $I_\xi$ and $I_\eta$, with same smoothness $C^3$ (or $C^{\infty}$) preserving the commuting property.
 \end{enumerate}
\end{definition}

Let $f$ be a $C^r$ multicritical circle map with irrational rotation number $\rho(f)$ and critical points $c_{0}, \dots, c_{N-1}$. For each critical point $c_j$, we can define a multicritical commuting pair in the following way: let $\widehat{f}$ be the lift of $f$ (under the universal covering $t \mapsto c_j\cdot\exp(2\pi i t)$) such that $0< \widehat{f}(0)<1$ (and note that $D\widehat{f}(0)=0$). For $n\geq 1$, let $\widehat{I}_{n}(c_j)$ be the closed interval in $\R$, containing the origin as one of its extreme points, which is projected onto $I_{n}(c_j)$. We define $\xi: \widehat{I}_{n+1}(c_j) \rightarrow \R$ and $\eta: \widehat{I}_{n}(c_j) \rightarrow \R$ by $\xi= T^{-p_{n}}\circ \widehat{f}^{q_{n}}$ and $\eta= T^{-p_{n+1}}\circ \widehat{f}^{q_{n+1}}$, where $T$ is the unit translation $T(x)=x+1$. Then the pair $(\eta|_{\widehat{I}_{n}(c_j)}, \xi|_{\widehat{I}_{n+1}(c_j)})$ is a multicritical commuting pair, that we denote by $(f^{q_{n+1}}|_{I_n(c_j)}, f^{q_n}|_{I_{n+1}(c_j)})$.\\

We restrict our attention to \emph{normalized} multicritical commuting pairs: for any given pair $\zeta=(\eta,\xi)$ we denote by $\widetilde{\zeta}$ the pair $(\widetilde{\eta}|_{\widetilde{I_{\eta}}}, \widetilde{\xi}|_{\widetilde{I_{\xi}}})$, where tilde means linear rescaling by the factor $1/|I_{\xi}|$. Note that $|\widetilde{I_{\xi}}|=1$ and $\widetilde{I_{\eta}}$ has length equal to the ratio between the lengths of $I_{\eta}$ and $I_{\xi}$. Equivalently $\widetilde{\eta}(0)=-1$ and $\widetilde{\xi}(0)=|I_{\eta}|/|I_{\xi}|=\xi(0)/\big|\eta(0)\big|$.

\begin{definition} We define the \emph{height} of the pair $\zeta=(\eta, \xi)$ as the natural number $a$ such that$$ \eta^{a+1}(\xi(0)) <0\leq \eta^{a}(\xi(0)),$$when such number exists, and we denote it by $\chi(\zeta)$. If such $a$ does not exist, that is, when $\eta$ has a fixed point, we define $\chi(\zeta)=\infty$.
\end{definition}

\begin{definition}\label{defren} Let $\zeta=(\eta, \xi)$ be a multicritical commuting pair with $(\xi \circ \eta)(0) \in I_{\eta}$ and $\chi(\zeta)=a< \infty$. We define the \emph{pre-renormalization} of $\zeta$ as the pair$$p\mathcal{R}(\zeta) =(\eta|_{[0, \eta^{a}(\xi(0)) ]} \ , \ \eta^{a}\circ \xi|_{I_{\xi}} ).$$
Moreover, we define the \emph{renormalization} of $\zeta$ as the normalization of $p \mathcal{R}(\zeta)$:$$\mathcal{R}(\zeta)= \left(\widetilde{\eta}|_{[0,\widetilde{\eta^{a}(\xi(0))} ]} \ , \ \widetilde{\eta^{a}\circ \xi}|_{\widetilde{I}_{\xi}}
  \right).$$
\end{definition}

If $\zeta$ is a multicritical commuting pair with $\chi(\mathcal{R}^{j}\zeta)< \infty$ for $0 \leq j \leq n-1$, we say that $\zeta$ is
 \textit{$n$-times renormalizable}, otherwise, if $\chi(\mathcal{R}^{j}\zeta)< \infty$  for all $j \in \nt$, we say that $\zeta$ is \textit{infinitely renormalizable}. In the last case, we define the \textit{rotation number} of the multicritical commuting pair $\zeta$, and denote it by $\rho(\zeta)$, as the irrational number whose continued fraction expansion is given by$$[\chi(\zeta), \chi(\mathcal{R}\zeta), \cdots, \chi(\mathcal{R}^{n}\zeta), \cdots ].$$
  
 Its normalization will be denoted by $\mathcal{R}_i^{n}f$, that is:$$\mathcal{R}_i^{n}f=\left(\widetilde{f}^{q_{n+1}}|_{\widetilde{I_n}(c_i)}, \widetilde{f}^{q_{n}}|_{\widetilde{I_{n+1}}(c_i)}\right).$$

Observe that $\rho(\mathcal{R}(\zeta)))=G(\rho(\zeta)))$, where $G$ is the Gauss map.

\section{Complex {\it a priori} bounds}\label{sec:complexbounds}
\subsection{Holomorphic commuting pairs}
We recall the definition of a holomorphic commuting pair given in \cite{Yam2019}, which generalizes the orginal definition of de~Faria \cite{dF2}. 

\begin{definition}\label{def:hcp}
 Given an analytic multicritical commuting pair $\zeta=(\eta|_{I_{\eta}}, \xi_{I_{\xi}})$, we say that it extends to a \emph{holomorphic commuting pair} $\mathcal{H}$, if there exist three simply-connected and $\R-$symmetric domains $D,U,V\subseteq \C$, whose intersections with the real line are denoted by $I_U=U \cap \R$, $I_V=V \cap \R$ and $I_D=D \cap \R$ and a simply connected $\R-$symmetric Jordan domain $\Delta$ that satisfy the following,
\begin{enumerate}
 \item the endpoints of $I_U$ and $I_V$ are critical points of $\eta$ and $\xi$, respectively;
 \item  $\overline{D}, \overline{U}, \overline{V}$ are contained in $\Delta$; $\overline{U}\cap \overline{V} = \{0\} \subseteq D$; the sets $U \setminus D, V\setminus D, D \setminus U$ and $D \setminus V$ are non-empty, connected and simply-connected; $I_{\eta} \subset I_U \cup \{0\}$, $I_{\xi}\subset I_V \cup \{0\}$;
 \item $U \cap \mathbb{H}$, $V \cap \mathbb{H}$ and  $D \cap \mathbb{H}$ are Jordan domains;
 \item the maps $\eta$ and $\xi$ have analytic extensions to $U$ and $V$, respectively, so that $\eta$ is a branched covering map of $U$ onto $(\Delta \setminus \R) \cup \eta(I_U)$, and $\xi$ is a branched covering map of $V$ onto $(\Delta \setminus \R)\cup \xi(I_V)$, with all the critical points of both maps contained in the real line;
 \item the maps $\eta:U \to \Delta$ and $\xi: V \to \Delta$ can be extended to analytic maps $\widehat{\eta}:U \cup D \to \Delta$ and $\widehat{\xi}:V \cup D \to \Delta$, so that the map $\nu=\widehat{\eta}\circ \widehat{\xi}= \widehat{\xi}\circ \widehat{\eta}$ is defined in $D$ and is a branched covering of $D$ onto $(\Delta\setminus \R)\cup \nu(I_D)$ with only real branched points.
\end{enumerate}
\end{definition}

\begin{figure}[!ht]
\centering
\psfrag{Delta}[][][1]{$\Delta$} 
\psfrag{V}[][][1]{$V$}
\psfrag{U}[][][1]{$U$}
\psfrag{D}[][]{$D$} 
\psfrag{fv}[][][1]{$\xi$}
\psfrag{fn}[][][1]{$\eta$}
\psfrag{fd}[][][1]{$\nu$} 
\psfrag{I}[][][1]{$I_{\xi}$} 
\psfrag{J}[][]{$I_{\eta}$} 
\psfrag{M}[][]{$I_{D}$} 
\psfrag{v}[][]{$I_{V}$} 
\psfrag{n}[][]{$I_{U}$} 
\includegraphics[width=4in]{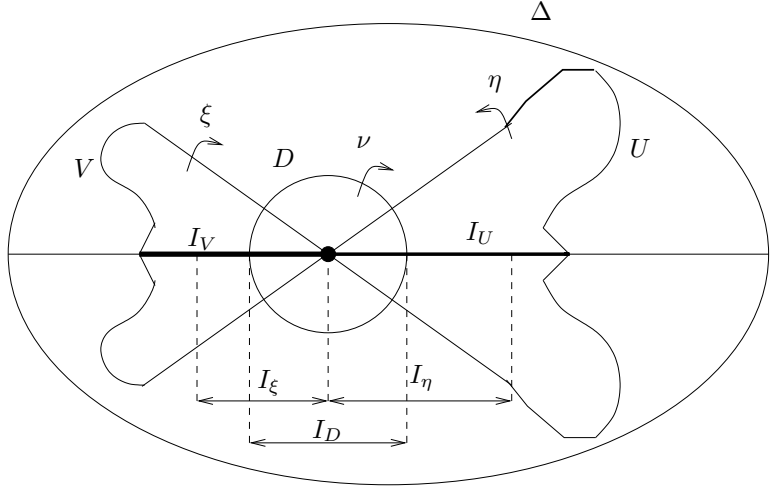}
  \caption{\label{fig:hcp} A holomorphic commuting pair.}
\end{figure}

\noindent
We shall identify a holomorphic pair $\cH$ with a triple of maps $\cH=(\eta,\xi, \nu)$, where $\eta\colon U\to\Delta$, $\xi\colon V\to\Delta$ and $\nu\colon D\to\Delta$. We shall also call $\zeta$ the {\it commuting pair underlying $\cH$}, and write $\zeta\equiv \zeta_\cH$. When no confusion is possible, we will use the same letters $\eta$ and $\xi$ to denote both the maps of the commuting pair $\zeta_\cH$ and their analytic extensions to the corresponding domains $U$ and $V$.

The sets $\Omega_\cH=D\cup U\cup V$ and $\Delta\equiv\Delta_\cH$ will be called \textit{the domain} and \textit{the range} of a holomorphic pair $\cH$. We will sometimes write $\Omega$ instead of $\Omega_\cH$, when this does not cause any confusion. 

We can associate to a holomorphic pair $\cH$  a piecewise defined map $S_\cH\colon\Omega\to\Delta$:
\begin{equation*}
S_\cH(z)=\begin{cases}
\eta(z),&\text{ if } z\in U,\\
\xi(z),&\text{ if } z\in V,\\
\nu(z),&\text{ if } z\in\Omega\setminus(U\cup V).
\end{cases}
\end{equation*}
De~Faria \cite{dF2} calls $S_\cH$ the {\it shadow} of the holomorphic pair $\cH$.

We can naturally view a holomorphic pair $\cH$ as three triples
$$(U,\xi(0),\eta),\;(V,\eta(0),\xi),\;(D,0,\nu).$$
We say that a sequence of holomorphic pairs converges in the sense of Carath{\'e}odory convergence, if the corresponding triples do.
We denote the space of triples equipped with this notion of convergence by $\hol$.

We let the {\emph modulus of a holomorphic commuting pair} $\cH$, which we denote by $\mod(\cH)$ to be the modulus of the largest annulus $A\subset \Delta$,
which separates $\CC\setminus\Delta$ from $\overline\Omega$.

\begin{definition}\label{H_mu_def}
For $\mu\in(0,1)$ let $\hol(\mu)\subset\hol$ denote the space of holomorphic commuting pairs
${\cH}:\Omega_{{\cH}}\to \Delta_{\cH}$, with the following properties:
\begin{enumerate}
\item $\mod (\cH)\ge\mu$;
\item {$|I_\eta|=1$, $|I_\xi|\ge\mu$} and $|\eta^{-1}(0)|\ge\mu$; 
\item $\dist(\eta(0),\partial V_\cH)/\diam V_\cH\ge\mu$ and $\dist(\xi(0),\partial U_\cH)/\diam U_\cH\ge\mu$;
\item {the domains $\Delta_\cH$, $U_\cH\cap\HH$, $V_\cH\cap\HH$ and $D_\cH\cap\HH$ are $(1/\mu)$-quasidisks.}
\item$\diam(\Delta_{\cH})\le 1/\mu$;
\end{enumerate}
\end{definition}

Let the {\it degree} of a holomorphic pair $\cH$ denote the maximal topological degree of the covering maps constituting the pair. Denote by $\hol^K(\mu)$ the subset of $\hol(\mu)$ consisting of pairs whose degree is bounded by $K$. The following is an easy generalization of Lemma 2.17 of \cite{Yam3}:  

\begin{lemma}
\label{bounds compactness}
For each $K\geq 3$ and $\mu\in(0,1)$ the space $\hol^K(\mu)$ is sequentially compact.
\end{lemma}

\noindent
We say that a real commuting pair $\zeta=(\eta,\xi)$ with
an irrational rotation number has
{\it complex {\rm a priori} bounds}, if there exists $\mu>0$ such that all renormalizations of $\zeta=(\eta,\xi)$ extend to 
holomorphic commuting pairs in $\hol(\mu)$.
The existense of complex {\it a priori} bounds is a key analytic issue 
 of renormalization theory. 

\begin{definition}\label{A_r_com_pair_def}
For  $S\subset\CC$ and $r>0$, we let $N_r(S)$ stand for the $r$-neighborhood of $S$ in $\CC$. 
For each $r>0$ we introduce a class $\cA_r$ consisting of pairs   
$(\eta,\xi)$ such that the following holds:
\begin{itemize}
\item $\eta$, $\xi$ are real-symmetric analytic maps defined in the domains 
$$U_r([0,1])\text{ and }U_{r|\eta(0)|}([\eta(0),0])$$
respectively, and continuous up to the boundary of the corresponding domains;
\item the pair $$\zeta\equiv (\eta|_{[0,1]},\xi|_{[\eta(0),0]})$$
is a multicritical commuting pair.
\end{itemize}
\end{definition}

\begin{remark}
  \label{rem:domains}
For simplicity, if $\zeta$ is as above, we will write $\zeta\in\cA_r$. But it is important to note that viewing our multicritical commuting pair $\zeta$ as an element of $\cA_r$ imposes restrictions on where we are allowed to iterate it. Specifically, we view such $\zeta$ as undefined at any point $z\notin U_r([0,\xi(0)])\cup U_{r|\eta(0)|}([0,\eta(0)])$ (even if $\zeta$ can be analytically continued to $z$). Similarly, when we talk about iterates of $\zeta\in\cA_r$ we iterate the restrictions $\eta|_{U_r([0,\xi(0)])}$ and $\xi|_{U_{r|\eta(0)|}([0,\eta(0)]}$.
In particular, we say that the first and second elements of $p\cR\zeta=(\eta^a\circ\xi,\eta)$ are defined in the maximal domains, where the corresponding iterates are defined in the above sense.
\end{remark}

\subsection{Complex bounds for pairs of bounded type}

We will denote $\cA_r^L$ the subset of $\cA_r$ consisting of pairs whose degree is bounded by $L$.
The following statement generalizes Theorem~3.2 of \cite{Yam2019} to all pairs of bounded type.
\begin{theorem}[{\bf Complex bounds for bounded type}]
  \label{thm:bounds}
Let $L\geq 3$, $B\in\NN$.  There exists a constant $\mu>0$  such that the following holds. For every positive real number $r>0$ and every pre-compact family $S\subset\mathcal A_r^L$ of multicritical commuting pairs, there exists $N=N(r, S)\in\NN$ such that if $\zeta\in S$ is
a commuting pair whose rotation number $\rho(\zeta)$ is of type bounded by $B$ then
$p\cR^n\zeta$ restricts to a holomorphic commuting pair $\cH_n:\Omega_n\to\Delta_n$ with $\Delta_n\subset U_r(I_\eta)\cup U_r(I_\xi)$, for all $n\geq N$.
Furthermore, the range $\Delta_n$ is a Euclidean disk,  and the appropriate affine rescaling of $\cH_n$ is in $\hol(\mu)$.
\end{theorem}

\noindent
Below we will give a proof of this theorem which generalizes the proof in \cite{Yam1999} and strengthens the proof in \cite{Yam2019}. For simplicity of notation, we will assume that the pair $\zeta$ is a pre-renormalization of a multicritical circle map $f$, that is
$$\zeta=(f^{q_n},f^{q_{n+1}}).$$
This will allow us to write explicit formulas for long compositions of terms $\eta$ and $\xi$, which will greatly streamline the exposition.
Theorem~\ref{thm:bounds} follows from Theorem~\ref{maintheorem} which we state below.
Let us further assume that $f$ has $N$ critical points, namely $c_0, c_1, \dots, c_{N-1}$, and that the critical point $c_0=0$ (which is the one at which we renormalize) has criticality equal to $d\in 2\NN+1$. Let $U$ be a $\TT$-symmetric annulus contained in the domain of analicity of $f$. We consider the dynamical partitions associated to the critical point $0$, and denote the fundamental domains of the $n-th$ partition by $I_n$ and $I_{n+1}$. Moreover, the $n-th$ renormalization of $f$ at $0$ is denoted by $\mathcal{R}^nf$.

Our main goal in this section is to prove the following result

\begin{theorem}\label{maintheorem}
 There exist universal positive constants $r_0,b,c$ such that the following holds. Let $f$ be as above and $r\geq r_0$. There exists $m_0=m_0(f,r) \in \nt$ such that for all $n\geq m_0$, if we denote by $\mathcal{R}^nf=(\eta,\xi)$ then there exists subdomains $V_{\eta} \supset I_{\eta}$, $V_{\xi} \supset I_{\xi}$ containing the origin, such that the maps $\eta, \xi$ are branched covering $V_{\eta} \to \{z \in \C: |z|<r \} \cap \C_{\eta(I_{\eta})}$ and $V_{\xi} \to \{z \in \C: |z|<r \} \cap \C_{\eta(I_{\xi})}$. Moreover, for each $z \in V_{\eta}\cap V_{\xi}$ we have
 \[
  |\mathcal{R}^{n}f(z)|\geq c|z|^{d} +b,
 \]
  where the left-hand side stands for $\eta$ on $V_{\eta}$ and $\xi$ on $V_{\xi}$. Finally, the constant $m_0$ can be chosen to be depending only on $r$ in any pre-compact family of maps in the compact-open topology on the annulus $U$.
\end{theorem}

Theorem \ref{maintheorem} is telling us that inverse branches of deep renormalizations, around the critical point $0$, behave as roots of degree $d$. Hence, these inverse branches map a large disk (containing $0$) well within itself, and therefore the modulus of the annuli between the disk and its image is bigger than a certain positive constant.

Following \cite{Yam1999}, Theorem \ref{maintheorem} will clearly follows from Lemma \ref{mainlemma} below.

\begin{lemma}\label{mainlemma}
There exist constants $B_1,B_2$ and $M>0$ such that for any $n\geq M$ the inverse branch $f^{-{q_{n+1}+1}}$ is well defined and univalent over $$\Omega_{n,M}=(D_{n-M}\setminus \R) \cup f^{q_{n+1}}(I_n),$$ and for any $z \in \Omega_{n,M}$ we have the following
\begin{equation}\label{ineqmainlemma}
 \frac{\dist(z_{-(q_{n+1}-1)}, f(I_n))}{|f(I_n)|} \leq B_1 \, \frac{\dist(z,I_n)}{|I_n|} + B_2.
\end{equation} 
\end{lemma}


\subsection{Proof of Lemma \ref{mainlemma}}
In this subsection we give a proof of Lemma \ref{mainlemma} based in the proof of \cite[Lemma 4.2]{Yam2019}. Let us introduce some notation. Let $M>n_0$, where $n_0$ is given by Corollary \ref{corollarybeau} and let $n \geq M$. We define $H_n$ to be the interval 
\begin{equation*}
 H_n= [f^{q_{n+1}}(0), f^{q_{n}-q_{n+1}}(0)],
\end{equation*}
and $D_n$ the Euclidean disk whose intersection with the real line is the interval $H_n$. Note that, by Corollary \ref{corollarybeau}, $\diam(D_n)=|H_n|\asymp |I_n|$. Also, consider the inverse orbit:
\begin{equation}\label{J-orbit}
J_0 = f^{q_{n+1}}(I_n), \ J_{-1}= f^{q_{n+1}-1}(I_n), \cdots , J_{-(q_{n+1}-1)} = f(I_n). 
\end{equation}
For any point $z \in D_M$, we say that
\begin{equation}\label{z-orbit}
 z_0 = z, z_{-1} , \cdots , z_{-(q_{n+1} -1)}
\end{equation}
is a corresponding inverse orbit if each $z_{-(k+1)}$ is obtained by applying to $z_{-k}$ a univalent
inverse branch of $f|_W$, where $W$ is a sub-interval of $J_{-(k+1)}$.

We will need four lemmas, the first one is the following result from \cite{Yam2019}.

\begin{lemma}\label{lemmadifeo} For each $n \geq1$ there exist $K_n \geq1$ and $\theta_n >0$ with $K_n \to 1$ and $\theta_n\to 0$ as $n \to \infty$ such that the following holds. Let $\theta \geq \theta_n$, and let  $0 \leq i < j \leq q_{n+1}$ be such that the restriction $f^{j-i} : f^{i}(I_n) \to f^{j}(I_n)$  is a diffeomorphism on the interior. Then the inverse branch $f^{-(j-i)}|_{f^{j}(I_n)}$ is well-defined
over $D_{\theta}(f^{j}(I_n))$ and maps it univalently into the Poincaré neighborhood $D_{\theta/K_n}(f^{i}(I_n))$.
\end{lemma}

Let $J=[a,b]$, for a point $z \in \overline{\C}_{J}$ we define the angle between $z$ and $J$, which is denoted by $\widehat{(z,J)}$, as the least of the angles between the intervals $[a,z],[b, z]$ and the corresponding rays $(a, -\infty), [b, +\infty)$ of the real line, measured in the range $0\leq \theta \leq \pi$.

Next result is the same as \cite[Lemma 4.7]{Yam2019}. We provide a more detailed proof for reader's convenience.

\begin{lemma}\label{lemmaineqonelevel}
Fix $n \geq M $, $\varepsilon_1>0$ and $B>0$. Let $0<i<k<q_{n+1}$ and consider two intervals of the inverse orbit given in \ref{J-orbit}, namely $J=J_{-i}$ and $J'=J_{-k}$. Let $z,z'$ be the corresponding points of the orbit \ref{z-orbit}. Assume that $\widehat{(z,J)}\geq \varepsilon_1$ and $\dist(z,J) \leq B\, |J|$. Then
 \begin{equation*}
  \frac{\dist(z',J')}{|J'|} \leq C\, \frac{\dist(z,J)}{|J|},
 \end{equation*}
for some constant $C=C(\varepsilon_1,B)>0$.
\end{lemma}

\begin{proof}[Proof of Lemma \ref{lemmaineqonelevel}]
  Since the orbit $\{J_{-i}\}_{0<i\leq q_{n+1}-1}$ forms part of the dynamical partition $\mathcal{P}_n(f)$, then there are at most one critical point for the iterate $f^{k-i}|_{J'}$. That critical point belongs to the interior or to the boundary of some interval $\Delta$ of the next dynamical partition $\mathcal{P}_{n+1}(f)$ in $J'$. Note that $|\Delta|\asymp |J'|$, see \cite[Lemma $4.2$]{EdF} and \cite[Proposition $4.1$]{EdF}. Therefore, there exists at most two intervals $I'_1,I'_2\subset J'$ comparable with $J'$ and such that $f^{k-i}: I'_j \to I_j$, for $j=1,2$, is a diffeomorphism. Let $D_{\theta_j}(I_j)$ be the smallest closed hyperbolic neighborhood enclosing $z$, for $j=1,2$. Observe that $\theta_j=\theta_j(\varepsilon_1,B)$, $ \diam(D_{\theta_j}(I_j)) \asymp \diam(D_{\theta_j}(J))$ and that there exists a constant $\tilde{C}_j=\tilde{C}_j(\varepsilon_1,B)>0$ such that $\diam(D_{\theta_j}(I_j)) \leq \tilde{C}_j \, \dist(z,I_j)$, see \cite[Lemma 2.1]{Yam1999}. By Lemma \ref{lemmadifeo} there exists $K_n>1$ such that $f^{-(k-i)}(D_{\theta_j}(I_j)) \subseteq  D_{\theta_j/K_n}(I'_j)$. Then for some $j\in \{1,2\}$
 {\small
 \begin{equation*}
 \dfrac{\dist(z', J')}{|J'|} \leq 
 \dfrac{C_1\, \diam(D_{\theta_j/K_n}(I'_j))}{|I'_j|} \leq
 \dfrac{C_2\, \diam ( f^{-(k-i)}(D_{\theta_j}(I_j)))}{|I'_j|} 
 \end{equation*}
 }
 where the constants $C_1,C_2$ are beau. The lemma follows since
 {\small
 \begin{equation*} 
 \dfrac{C_2\,  \diam ( f^{-(k-i)}(D_{\theta_j}(I_j)))}{|I'_j|}
 \leq
 \dfrac{C_3\, \diam(D_{\theta_j}(I_j))}{|I_j|}
  \leq \dfrac{C_4 \, \dist(z, J)}{|I_j|}
  \leq \dfrac{C_5\, \dist(z,J)}{|J|},
 \end{equation*}
 }
 where $C_3$ is beau and $C_4, C_5$ depend on $\varepsilon_1$ and $B$.
 \end{proof}

Next result is borrowed from \cite[Page 345, Lemma 2.2]{dFdM2}.

\begin{lemma} \label{lemmaJ-orbitcontained} Let $n > M$ and consider the inverse orbit defined by \ref{J-orbit}. Given $m$ with $n>m \geq M$, let $P_0,  \dots P_{-k}$ be the moments in the backward orbit \ref{J-orbit} of $I_n$ before the first return to $I_{m+1}$ such that $P_{-i}\subseteq I_m$. Then $k=a_{m+1}$, $P_0 \subseteq I_{m+2}$ and
\[
 P_{-i}\subseteq f^{q_m+(a_{m+1}-i)q_{m+1}}(I_{m+1}).
\]
\end{lemma}

\begin{figure}[!ht]
\centering
\psfrag{I}[][][1]{$I_{m+1}$} 
\psfrag{J}[][][1]{$I_{m}$}
\psfrag{K}[][][1]{$I_{m+2}$}
\psfrag{L}[][]{$I_{m+1}^{q_{m+2}-q_{m+1}}$} 
\psfrag{M}[][][1]{$I_{m+1}^{q_m}$}
\psfrag{P}[][][1]{$P_0$}
\psfrag{Q}[][][1]{$P_{-a_{m+1}}$} 
\psfrag{c}[][][1]{$c$} 
\psfrag{...}[][]{$\Large{\dots}$} 
\includegraphics[width=4.5in]{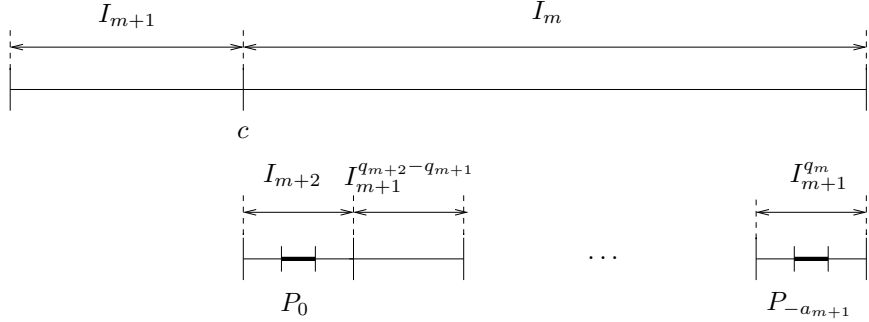}
  \caption{\label{fig:lemmaJ-orbitcontained} The moments in the backward orbit \ref{J-orbit} of $I_n$ before the first return to $I_{m+1}$.}
\end{figure}

Next result and its proof is a version of \cite[Lemma 4.8]{Yam2019}, for bounded type rotation numbers. We remark that this result and its proof is the main difference in the proof of complex {\it a priori} bounds for multicritical circle maps with bounded type rotation number and multicritical circle maps with irrational quadratic rotation number.

Since we are interested in following the orbit \eqref{J-orbit} of $I_n$ by the inverse branch of the map $f^{q_{m+1}}$ in $D_m$, we decompose the inverse branch by a finite composition of diffeomorphisms and inverse images of $f$. 

\begin{lemma}\label{lemmaineqseverallevels}
  There exists $\varepsilon_2>0$ such that for each $n>m$ the following holds. Let $J=J_{-k}$, $J'=J_{-k-q_{m+1}}$ be two consecutive returns of the backward orbit \ref{J-orbit} of $I_n$, before the first return to $I_{m+1}$, and let $\zeta,\zeta'$ be the corresponding points of the orbit \ref{z-orbit}. Suppose that $\zeta \in D_m$, then either $\zeta' \in D_m$ or the following holds. There is $k<s\leq k+q_{m+1}$ such that  $|J_{-s}|>C_0|J|$ and for the point $z_{-s}$ in the orbit \ref{z-orbit} 
  we have $\widehat{(z_{-s},J_{-s})}>\varepsilon_2$,  Moreover, $\widehat{(\zeta',J')}>\varepsilon_2$ and hence $\dist(\zeta', J')<C\,|I_m|$, where $C$ is beau.
  \end{lemma}

\begin{proof}[Proof of Lemma \ref{lemmaineqseverallevels}]
  Firstly, replacing $f$ by its renormalization if need be, we can ensure that every element of the dynamical partition of level $\geq M$ contains at most one critical value of $f$.  Now, if the iterate $f^{q_{m+1}}$ does not have a critical point in the interior of  the interval $I_m$, then we are in the situation of \cite[Lemma 4.2]{Yam1999} and the same proof applies. Therefore, let us assume that there are critical points of $f^{q_{m+1}}$ in the interior of $I_m$. For simplicity, let us assume that there is only one such point, otherwise the argument below will need to be repeated at most $N-1$ times (recall that $N$ is the number of critical points of $f$).

  Let us denote by $\beta$ the critical value of $f^{q_{m+1}}$ in the interior of $f^{q_{m+1}}(I_m)$; there exists $\ell<q_{m+1}$ such that $\beta =f^{\ell}(c_1)$, for $c_1\neq 0$ being a critical point of $f$ with criticality $d_1\in2\NN+1$. Below we will distinguish two scenarios. In the first one, the distance between $\beta$  and $J$ is commensurable with the size of $I_m$. Then we are in the case described in \cite[Lemma 4.8]{Yam2019}, and we proceed accordingly.

  In the other case, the pull-back  $J_{-k-(\ell-1)}\mapsto J_{-k-\ell}$ will factor as $\phi\circ s$ where $w\mapsto s(w)$ is a root of degree $d_1$, and $\phi$ is univalent, and $T\equiv \phi(J_{-k-(\ell-1)})$ is near to $0$. If the point $z_{-k-\ell}$ in the orbit \ref{z-orbit} ``jumps'' at some definite angle from the real line, then the relative distance
$\dist(z_{-k-\ell},J_{-k-\ell})/|J_{-k-\ell}|$ will not increase -- it 
will actually {\it decrease} up to a constant by a root of degree $d_1$\footnote{It is helpful to think here what happens in the limiting situation, that is, when $\beta$ is one of the endpoints of $\cP_m$. Then $f$ renormalizes to a commuting pair with a critical point of criticality $d\times d_1$ at $0$ and the estimate in Theorem~\ref{maintheorem} is {\it improved}.}. Hence the argument can be completed using Lemma~\ref{lemmaineqonelevel}.

We proceed with a formal discussion below. Let ${\Delta}$ be the interval of $\mathcal{P}_{m+1}$ containing $\beta$.
  Note that by our assumption on $\beta$, is not possible to have ${\Delta}=I_{m+1}^{q_m}$. Therefore there are two cases for ${\Delta}$: either ${\Delta}\subseteq I_m \setminus (I_{m+2} \cup I_{m+1}^{q_m})$ or ${\Delta}=I_{m+2}$. 
 \begin{enumerate}
 \item [1)] Let ${\Delta}\subseteq I_m\setminus (I_{m+2} \cup I_{m+1}^{q_m})$. In this case, there exists $r>0$ which is beau commensurable with $|I_m|$ such that
   $$\Delta_1\equiv U_r(\Delta)\subset D_m\text{ and }\Delta_1\cap U_r(f^{q_m}(0))=\emptyset.$$
   Let us observe that by our assumptions $J \subseteq \Delta_1$. In the case when $\zeta \notin \Delta_1$,  we are in the same situation as in \cite[Lemma 4.8]{Yam2019}, and the argument there applies {\it verbatim}. Therefore, let us assume that $\zeta\in\Delta_1$. 

Note that $J'\subset [f^{-q_{m+1}(0)},f^{q_m-q_{m+1}}(0)]$ and the endpoint $f^{q_m-q_{m+1}}(0)$ is a critical value of the iterate $f^{q_{m+1}}$. This easily implies (cf. \cite[Figure 3]{Yam2019}) that there exists a beau $\veps>0$ such that if $\widehat{(z_{-j},J_{-j})}>\veps$ for all $j$ between $k$ and $k+q_{m+1}$, then $\zeta'\in D_m$.
As we are pulling back by the iterate $f^{q_{m+1}}$, our interval $H_m$ will go through a critical value of $f$ three times: twice through $f(0)$, for the pullbacks by $f^{q_m-1}$ and $f^{q_{m+1}-1}$, and once through $f(c_1)$ via pullback by $f^{\ell-1}$.

Suppose $\widehat{(\zeta,J)}<\veps$ (otherwise, we would be done by Lemma~\ref{lemmaineqonelevel}). 
By considerations of Koebe Distortion Theorem, there exist $\veps_2>0$ beau and also beau constants of commensurability such that one of the following possibilities holds:
\begin{itemize}
\item[a)] $\zeta'\in D_m$;
\item[b)] $\dist(z_{-k-(q_{m}-1)},J_{-k-(q_m-1)}) \asymp |f^{-(q_m-1)}(I_m)|$, \\ $\widehat{(z_{-k-(q_m-1)},J_{-k-(q_m-1)})}<\veps$ \, and $\widehat{(z_{-k-q_m},J_{-k-q_m})}>\veps_2$;
\item[c)] $\dist(z_{-k-(q_{m+1}-q_m-1)},J_{-k-(q_{m+1}-q_m-1)}) \asymp |f^{-(q_{m+1}-q_m-1)}(I_m)|$,\\ $\widehat{(z_{-k-(q_{m+1}-q_m-1)},J_{-k-(q_{m+1}-q_m-1)})}<\veps$ \, and \\
$\widehat{(z_{-k-(q_{m+1}-q_m)},J_{-k-(q_{m+1}-q_m)})}>\veps_2$;
\item[d)] $\dist(z_{-k-(\ell-1)},J_{-k-(\ell-1)}) \asymp |f^{-(\ell-1)}(I_m)|$, \\$\widehat{(z_{-k-(\ell-1)},J_{-k-(\ell-1)})}<\veps$ \, and $\widehat{(z_{-k-\ell},J_{-k-\ell})}>\veps_2$.
\end{itemize}
Since we have assumed that $J$ is far from the critical values $f^{q_m}(0)$,
$f^{q_{m+1}-q_m}(0)$, in case b) we have $|J_{-k-q_m}| \asymp |J'|$, and the statement follows from Lemma~\ref{lemmaineqonelevel}.
Case c) is handled in the same way. It remains to discuss case d). Now, since we have assumed that $\zeta$ is close to $\beta$, we will have
$$\frac{\dist(z_{-k-\ell},J_{-k-\ell})}{|J_{-k-\ell}|} \, < \, C_1 \, \sqrt[d_1]{\frac{\dist(z_{-k-(\ell-1)},J_{-k-(\ell-1)})}{|J_{-k-(\ell-1)}|}}+C_2,$$
for $C_1,C_2$ beau constants. The claim follows from Lemma~\ref{lemmaineqonelevel}, see Figure \ref{fig:cased}
\begin{figure}[!ht]
\centering
\psfrag{0}[][][1]{$0$}
\psfrag{Dm}[][][1]{$D_{m}$}
\psfrag{J-k-l-1}[][]{$J_{-k-(\ell-1)}$} 
\psfrag{B}[][][1]{$f^{q_{m+1}-\ell}$}
\psfrag{A}[][][1]{$f^{\ell-1}$} 
\psfrag{w}[][][1]{$c_1$}
\psfrag{ADm}[][][1]{$\tilde{D}_m$} 
\psfrag{j}[][][1]{$f(c_1)$} 
\psfrag{phicircs}[][][1]{$\phi \circ s$} 
\psfrag{tildeJ}[][][1]{$J'$}
\psfrag{hatJ}[][][1]{$J_{-k-\ell}$} 
\psfrag{hatDm}[][][1]{$\hat{D}_m$} 
\psfrag{delta1}[][][1]{$\Delta_1$}
\psfrag{tildeDm}[][][1]{$D'_m$} 
\psfrag{J}[][][1]{$J$}  
\includegraphics[width=4.6in]{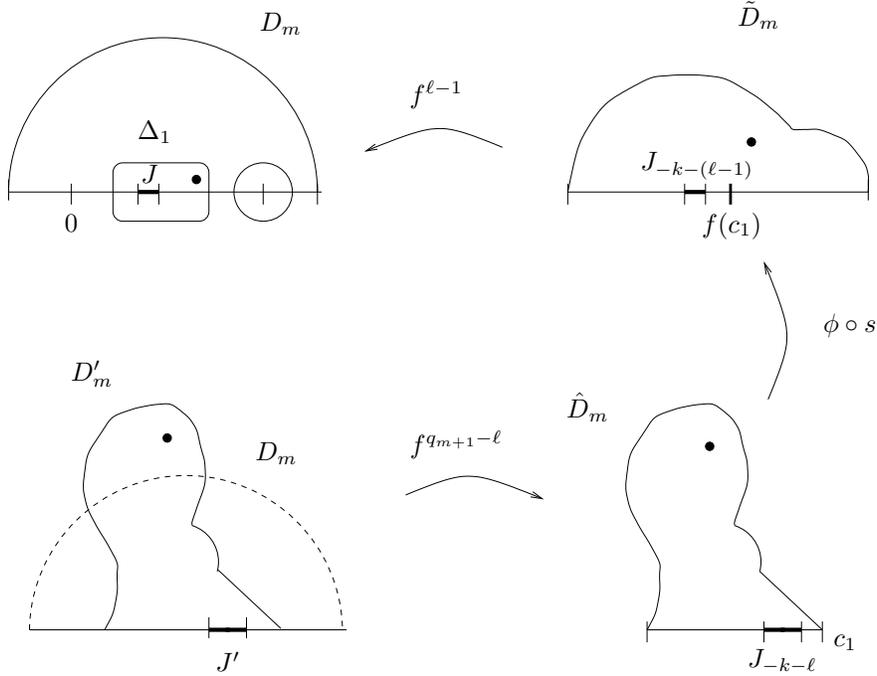}
  \caption{\label{fig:cased} Proof of Lemma \ref{lemmaineqseverallevels}, case d).}
  \end{figure}

  \item [2)] If ${\Delta}=I_{m+2}$ then for $J$ as in the statement we can have that $J$ is between $0$ and $\beta$, or $\beta \in J$, or $J$ is between $\beta$ and $f^{q_{m+2}}(0)$. In any case, we can repeat the strategy used in item 1) and obtain the result.
  \end{enumerate} 
 \end{proof}
 
 Next result will let us to obtain the step of induction that we will use in the proof of Main Lemma. We refer the reader to \cite[Lemma 4.9]{Yam2019}.

\begin{lemma}\label{lemmaineqinductionstep}
There exists $\varepsilon_3>0$ such that for each $n>m$ we have the following. Let $J$ be the last return of the backward orbit \ref{J-orbit} to $I_n$ before the first return to $I_{m+1}$. Let $J'$ and $J''$ be the first two returns of \ref{J-orbit} to $I_{m+1}$ and $\zeta,\zeta'$ be the corresponding moments in the backward orbit \ref{z-orbit}, in other words, $\zeta=f^{q_m}(\zeta')$ and $\zeta'=f^{q_{m+2}}(\zeta'')$. Suppose that $\zeta \in D_m$, then either $\zeta'' \in D_{m+1}$, or $\widehat{(\zeta'', I_{m+1})}>\varepsilon_3$ (and $\dist(\zeta'', J'')< C|I_{m+1}|$ where $C$ is beau).  
\end{lemma}

\begin{proof}[Proof of Lemma \ref{lemmaineqinductionstep}] 
By our assumption in the proof of Lemma \ref{lemmaineqseverallevels}, if $\Delta$ is the interval belonging to $\mathcal{P}_{n+1}$ containing the point $\beta=f^{\ell}(c_1)$, then $\Delta\neq I_{m+1}^{q_m}$.  We have two cases; $\Delta=I_{m+2}$ or $\Delta\neq I_{m+2}$. In Figure \ref{fig:lemmainductionstep} we show the case $\Delta=I_{m+2}$ and $q_m<\ell$, the other cases are similar. Let $D^{''}_m=f^{-(q_m-q_{m+2})}(D_m)$. Using an analogous argument used in the proof of Lemma \ref{lemmaineqseverallevels}, that is, decomposing the iterate $f^{q_{m+2}-q_m}$ as compositions of diffeomorphisms and one iterate of $f$ we get the following: there exist $\varepsilon_3$ beau and $I \subseteq H_{m+1}$ with $|I|\asymp |I_{m+1}|$,  such that $D^{''}_m \subseteq D_{m+1} \cup D_{\varepsilon_3}(I)$, see Figure \ref{fig:lemmainductionstep} below.  
\begin{figure}[!ht]
\centering
\psfrag{Dm}[][][1]{$D_{m}$}
\psfrag{Dm+1}[][][1]{$D_{m+1}$}
\psfrag{A}[][]{$f^{q_m}$} 
\psfrag{B}[][][1]{$f^{\ell-q_m}$}
\psfrag{C}[][][1]{$f^{q_m+q_{m+1}-\ell}$}
\psfrag{L}[][][1]{$f^{q_{m+2}-q_{m+1}}$}
\psfrag{a}[][][1]{$f^{q_{m+1}}(0)$} 
\psfrag{e}[][][1]{$\ f^{q_m-q_{m+1}}(0)$}
\psfrag{v}[][][1]{$f^{q_{m+1}-q_{m+2}}(0)$}
\psfrag{p}[][][1]{$f^{-q_m}(\beta)$} 
\psfrag{beta}[][][1]{$\beta$}
\psfrag{Dm'}[][][1]{$D'_m$} 
\psfrag{j}[][][1]{$0$}
\psfrag{f}[][][1]{$f^{q_{m+1}-q_m}(0)$}
\psfrag{k}[][][1]{$c_{1}$}
\psfrag{w}[][][1]{$f^{-\ell+q_m}(0)$}
\psfrag{r}[][][1]{$0$}
\psfrag{s}[][][1]{$f^{-q_{m+1}}(0)$} 
\psfrag{Dm''}[][][1]{$D^{''}_m$} 
\psfrag{tildeDm}[][][1]{$\tilde{D}_m$} 
\psfrag{hatDm}[][][1]{$\widehat{D}_m$} 
\psfrag{Im+2}[][][1]{$I_{m+2}$}
\psfrag{Im+1qm}[][][1]{$I_{m+1}^{q_m}$}  
\includegraphics[width=4.8in]{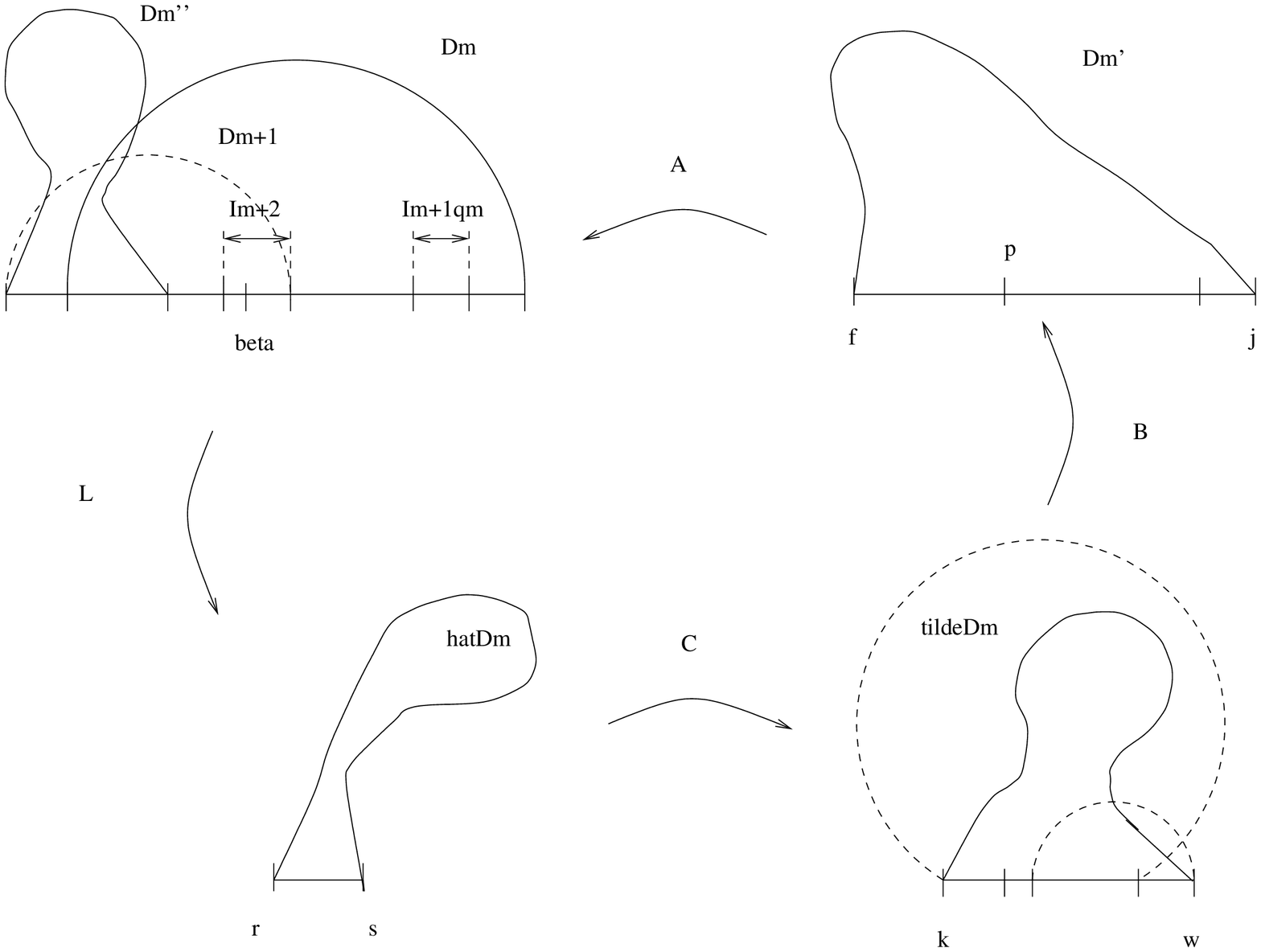}
  \caption{\label{fig:lemmainductionstep} Proof of Lemma \ref{lemmaineqinductionstep}.}
  \end{figure}
  
  Then either $\zeta'' \in D_{m+1}$ or $\widehat{(\zeta'', I_{m+1})}>\varepsilon_3$. In the last case, by equation \eqref{eq:diamD} we have
\[
 \dist(\zeta'', J') \leq \diam(D_{\varepsilon_3}(I)) = C\,|I|\asymp |I_{m+1}|,
\]
for $C=C(\varepsilon_3)>0$ and therefore beau.
 \end{proof}

 Now, with Lemma \ref{lemmaineqseverallevels} and Lemma \ref{lemmaineqinductionstep} at hand we proceed to prove our Main Lemma. 
 
\begin{proof}[Proof of Lemma \ref{mainlemma}]
Let $z \in D_M$. Let $m$ be the largest number such that $z \in D_m$. We always will have two cases, and in each case we will get the inequality \eqref{ineqmainlemma} for each $n \geq m$. Let $P_0, \dots, P_{-k}$ be the consecutive returns of the backward orbit \ref{J-orbit} of $I_n$ to $I_m$ before the first return to $I_{m+1}$ and denote by $z=\zeta_0, \dots \zeta_{-k}=\zeta'$ the corresponding points of the orbit \ref{z-orbit}. By Lemma \ref{lemmaineqseverallevels} there exist beau constants $C>0$ and $\varepsilon_2>0$ such that 
$\widehat{(\zeta', J')}>\varepsilon_2$ and $\dist(\zeta',P_{-k})<C\, |I_m|$, or $\zeta' \in D_m$. In the first case, by Lemma \ref{lemmaineqonelevel} we obtain inequality \eqref{ineqmainlemma}. In the second case, we consider the point $\zeta''$ corresponding to the second return of the orbit \ref{z-orbit} to $I_{m+1}$. By Lemma \ref{lemmaineqinductionstep}, there exist beau constants $\varepsilon_3>0$ and $C>0$ such that either $\widehat{(\zeta'',I_{m+1})}>\varepsilon_3$ and $\dist(\zeta'',I_{m+1})<C\, |I_{m+1}|$, or $\zeta'' \in D_{m+1}$. In the first case, we obtain the inequality \eqref{ineqmainlemma} by Lemma \ref{lemmaineqonelevel}. In the second case, we repeat the previous argument this time for $m+1$ instead $m$.
\end{proof}

\section{Applications to bi-cubic maps}
Let us now specialize to the case when a multicritical circle map $f$ has exactly two critical points, both of which are of criticality $3$. We will call such maps {\it bi-cubic}, and place one of these points at $0$ to fix the ideas; we will denote the other critical point by $c$. The renormalizations of such map will then be defined with respect to the critical point at $0$.

The following is a generalization of \cite[Theorem 2.8]{Yam2019}, the proof applies  {\it verbatim}, and will be omitted:
\begin{theorem}
  \label{th:renconv}
  For each $K\in\NN$ there exists $\lambda\in(0,1)$ such that the following holds.
  Suppose $f$ and $g$ are two bi-cubic critical circle maps with the same signature, and assume that $\rho(f)=\rho(g)$ is of a type bounded by $K$. Then
    $$\dist(\cR^jf,\cR^jg)=o(\lambda^j)$$
  in the uniform norm on a neighborhood of their intervals of definition.
\end{theorem}

\noindent
The following result is Main Theorem in \cite{EG}

\begin{theorem}\label{rigidityEG}
There exists a full Lebesgue measure set $\mathcal{A}\subset(0,1)$ of irrational numbers (which includes the set of bounded type numbers) with the following property. Let $f$ and $g$ be $C^3$ multicritical circle maps with the same signature and such that its common rotation number belongs to the set $\mathcal{A}$. If the renormalizations of $f$ and $g$ around corresponding critical points converge together exponentially fast in the $C^1$ topology, then $f$ and $g$ are conjugate to each other by a $C^{1+\alpha}$ diffeomorphism.  
\end{theorem}

%
Therefore, Theorem~\ref{th:renconv} and 
Theorem \ref{rigidityEG} imply the following result,
\begin{theorem}
  \label{th:rigidity} Let $K\in\NN$. There exists $\alpha>0$ such that the following holds. 
  Suppose $f$ and $g$ are two bi-cubic circle maps whose signatures are the same, and furthermore, $\rho(f)=\rho(g)$ is of a type bounded by $K$. Then
  $f$ and $g$ are $C^{1+\alpha}$ conjugate on $\TT$.
  \end{theorem}


\end{document}